\documentclass[12pt]{article}
\usepackage{enumerate}
\usepackage{amsmath,amssymb,amsthm, mathrsfs, mathtools}
\usepackage{latexsym}
\usepackage{graphics,graphicx}
\usepackage{hyperref}
\usepackage[bf, small]{titlesec}
\usepackage{authblk} 
\usepackage{soul}
\usepackage{kotex}
\usepackage{lineno}

\theoremstyle{plain}
\newtheorem{Thm}{Theorem}[section]

\newtheorem{Lem}[Thm]{Lemma}
\newtheorem{Prop}[Thm]{Proposition}
\newtheorem{Cor}[Thm]{Corollary}

\theoremstyle{definition}

\usepackage{standalone}
\usepackage{tikz}
\usetikzlibrary{arrows,positioning,matrix,fit,backgrounds,shapes,shapes.geometric, intersections}

\usetikzlibrary{shapes.callouts,decorations.pathmorphing} 
\usepackage{calc} 
\usetikzlibrary{decorations.markings} 
\tikzstyle{vertex}=[circle, draw, inner sep=0pt, minimum size=6pt] 

\usetikzlibrary{arrows,matrix} 
\newtheorem*{claim1}{Claim A}

\newcommand{\olr}[1]{\overleftrightarrow{#1}}

\title{Digraphs in which every $t$ vertices share exactly $\lambda$ out-neighbors and exactly $\lambda$ in-neighbors}

\author[1]{Hojin Chu}
\author[1]{Suh-Ryung Kim}

\affil[1]{\footnotesize Department of Mathematics Education, Seoul National University, Seoul 08826, Rep. of Korea}
\affil[ ]{\footnotesize\textit{ghwls8775@snu.ac.kr, srkim@snu.ac.kr}}
\date{}

\definecolor{LemonChiffon}{rgb}{100, 98, 80}
\definecolor{myblue}{rgb}{0,0.4,0.8}
\definecolor{orange}{rgb}{1, 0.4, 0}
\definecolor{mygreen}{rgb}{0, 0.8, 0}
\definecolor{myred}{rgb}{204, 0, 0}
\definecolor{violet}{RGB}{0.4,0.2,1}
\definecolor{brown}{rgb}{0.6, 0.4, 0}

\newcounter{statement}

\begin{document}
\maketitle

 \begin{abstract}
 In this paper, we introduce the notion of two-way $(t,\lambda)$-liking digraphs as a way to extend the results for generalized friendship graphs. 
 A two-way $(t,\lambda)$-liking digraph is a digraph in which every $t$ vertices have exactly $\lambda$ common out-neighbors and $\lambda$ common in-neighbors. 
 We first show that if $\lambda \ge 2$, then a two-way $(2,\lambda)$-liking digraph of order $n$ is $k$-diregular for a positive integer $k$ satisfying the equation $(n-1)\lambda=k(k-1)$.
 This result is comparable to the result by Bose and Shrikhande in 1969 and actually extends it. 
Another main result is that if $t \ge 3$, then the complete digraph on $t+\lambda$ vertices is the only two-way $(t,\lambda)$-liking digraph.
 This result can stand up to the result by Carstens and Kruse in 1977 and essentially extends it.
 In addition, we find that two-way $(t, \lambda)$-liking digraphs are closely linked to symmetric block designs and extend some existing results of $(t, \lambda)$-liking digraphs.
  \end{abstract}



\section{Introduction}
In this paper, for graph-theoretical terminology and notations not defined, we follow \cite{bang2018classes} and \cite{bondy2010graph}.
Neither graphs nor digraphs in this paper have loops, multiple edges, or multiple arcs.
In 1966, Erd\"{o}s {\it et al.} \cite{erd1966r} introduced and proved the Friendship Theorem.
The Friendship Theorem can be stated, in graph-theoretical terms, as follows: if any pair of vertices in a graph has exactly one common neighbor, then there exists a vertex adjacent to the others.
The graph described in the Friendship Theorem is referred to as a ``friendship graph".
A {\it friendship graph} is a graph in which every pair of vertices has exactly one common neighbor. 
Many  variants  of friendship graphs  have  been  studied. 
As one of those variants, a {\it generalized friendship graph} is a graph in which every $t$ vertices have exactly $\lambda$ common neighbors for some positive integers $t$ and $\lambda$. 
In \cite{bose1969graphs}, it was shown that a generalized friendship graph is regular if $t=2$ and $\lambda \ge 2$.
In \cite{CARSTENS1977286}, it was shown that a generalized friendship graph is isomorphic to the complete graph on $t+\lambda$ vertices if $t\ge 3$. 
In \cite{DELORME1984261}, an infinite generalized friendship graph is studied.  

There is also a variation of the Friendship Theorem in terms of digraph.
A digraph is called a
{\it $(t,\lambda)$-liking digraph} if every $t$ vertices have exactly $\lambda$ common out-neighbors where $t$ and $\lambda$ are positive integers.
In 1974,
 M\"{u}ller and Pelant \cite{muller1974strongly} studied the case where tournaments are $(t,\lambda)$-liking digraphs and showed the non-existence of such tournaments for $t\geq 3$.
In 1975, Plesn\'{i}k \cite{plesnik1975graphs} characterized the $(t,1)$-liking digraphs for each integer $t\ge 3$ by proving that every $(t,1)$-liking digraph with $t\ge 3$ is the complete digraph on $t+1$ vertices.
Recently, Choi {\it et al.} \cite{choi2023digraph} completely characterized $(2,1)$-liking digraphs.
In the follow-up paper \cite{choi2024generalizedlikingdigraph}, the authors extended their results and studied the case where $(t,\lambda)$-liking digraphs are complete.
 
While seeking a way to extend the results for generalized friendship graphs in \cite{bose1969graphs}, \cite{CARSTENS1977286}, and \cite{sudolsky1978generalization} (mentioned above) into a digraph version, we were led to introduce the notion of ``two-way $(t,\lambda)$-liking digraphs".
A digraph is called a
{\it two-way $(t,\lambda)$-liking digraph} if every $t$ vertices have exactly $\lambda$ common out-neighbors and exactly $\lambda$ common in-neighbors where $t$ and $\lambda$ are positive integers.
Note that the digraph obtained from a generalized friendship graph by replacing each edge with a directed cycle of length two is a two-way $(t,\lambda)$-liking digraph.
Accordingly, the existing results for generalized friendship graphs can be derived from our main results described below.

In this paper, we characterize the two-way $(t,\lambda)$-liking digraphs.
By the definition of a two-way $(t,\lambda)$-liking digraph, it is obvious that a two-way $(2,1)$-liking digraph is a $(2,1)$-liking digraph.
Choi {\it et al.}~\cite{choi2023digraph} showed that a $(2,1)$-liking digraph is a two-way $(2,1)$-liking digraph.
Theorem~\ref{thm:friendship_digraph} which is the primary result in \cite{choi2023digraph} is a complete characterization of $(2,1)$-liking digraphs and therefore the result is true for two-way $(2,1)$-liking digraphs.
Thus it remains to characterize the two-way $(t,\lambda)$-liking digraphs with $(t,\lambda) \neq (2,1)$, which is achieved by the following two theorems.

\begin{Thm}\label{thm:main1}
	If $\lambda \ge 2$, then a two-way $(2,\lambda)$-liking digraph of order $n$ is $k$-diregular where the positive integer $k$ satisfies
	\[(n-1)\lambda=k(k-1). \]
\end{Thm}

\begin{Thm}\label{thm:main2}
	If $t\ge 3$, then the complete digraph on $t+\lambda$ vertices is the only two-way $(t,\lambda)$-liking digraph.
\end{Thm}

Going a little further from Theorem~\ref{thm:main1}, we present Proposition~\ref{prop:design implies}.
Choi {\it et al.} showed that an $(n,k,\lambda)$-SBIBD can be constructed from a $k$-diregular $(2,\lambda)$-liking digraph of order $n$.
That is, the existence of a $k$-diregular $(2,\lambda)$-liking digraph of order $n$ guarantees the existence of an $(n,k,\lambda)$-SBIBD.
Since a two-way $(2,\lambda)$-liking digraph is a  $(2,\lambda)$-liking digraph,
we can say that an $(n,k,\lambda)$-SBIBD exists if a two-way $(2,\lambda)$-liking digraph of order $n$ exists.
We show that the converse is true for $n\ge 2\lambda$.

Theorem~\ref{thm:main2} extends one of the main theorems in Choi {\it et al.} \cite{choi2024generalizedlikingdigraph}.
By the definition of a two-way $(t,\lambda)$-liking digraph, it is obvious that a two-way $(t,\lambda)$-liking digraph is a $(t,\lambda)$-liking digraph.
Under this circumstance, it is natural to ask for which $t$ and $\lambda$ a $(t,\lambda)$-liking digraph being a two-way $(t,\lambda)$-liking digraph.
For example, it is true when $t=2$ and $\lambda=1$ by \cite{choi2023digraph}. 
In the follow-up paper \cite{choi2024generalizedlikingdigraph}, the authors showed that for $t\ge \lambda+2$, the complete digraph on $t+\lambda$ vertices is the only $(t,\lambda)$-liking digraph.
It is obvious that the complete digraph on $t+\lambda$ vertices is a two-way $(t,\lambda)$-liking digraph.
Thus a $(t,\lambda)$-liking digraph is a two-way $(t,\lambda)$-liking digraph if $t\ge \lambda+2$.
In this paper, we extend the result by showing that it is true for the cases of $t \ge \lambda+1$ (Corollary~\ref{cor:lambda+1_twoway}). 
Finally, this result, together with Theorem~\ref{thm:main2}, extends Theorem~\ref{thm:prev_main1} (Theorem~\ref{thm:lambda+1_complete}).

	\section{Preliminaries}\label{sec:pre}

 A {\it fancy wheel digraph} is obtained from the disjoint union of directed cycles by adding one vertex with arcs to and from each vertex on the cycles.
    A {\it $k$-diregular digraph} is a digraph in which each vertex has outdegree $k$ and indegree $k$ for a positive integer $k$.
	A digraph is said to be {\it diregular} if it is a $k$-diregular digraph for some positive integer $k$.

\begin{Thm}[\cite{choi2023digraph}]\label{thm:friendship_digraph}	
A $(2,1)$-liking digraph is a fancy wheel digraph or is $k$-diregular of order $k^2-k+1$ for some integer $k\ge 2$.
\end{Thm}

A {\it balanced block design} consists of a set $X$ of $v \geq 2$ elements, called {\it varieties}, and a collection of $b>0$ subsets of $X$, called {\it blocks}, such that the following conditions are satisfied:
	\begin{itemize}
	\item Each block consists of exactly the same number $k$ of varieties where $k>0$;
	\item Each variety appears in exactly the same number $r$ of blocks where $r>0$;
	\item Each pair of varieties appears simultaneously in exactly the same number $\lambda$ of blocks where $\lambda>0$.
	\end{itemize}
A balanced block design with $k < v$ is called a {\it balanced incomplete block design} since each block has fewer varieties than the total number of varieties. 
Such a design is also called a {\it $(b,v,r,k,\lambda)$-BIBD}.
A $(b,v,r,k,\lambda)$-BIBD is said to be {\it symmetric} if $b=v$ and $r=k$.
A symmetric $(b,v,r,k,\lambda)$-BIBD is termed as a {\it$(v,k,\lambda)$-SBIBD}.
It is a well-known fact that in a $(v,k,\lambda)$-SBIBD, any two blocks have exactly $\lambda$ varieties in common.

	Let $A_1,A_2,\ldots,A_n$ be sets. A {\it (complete) system of distinct representatives} is a sequence $(a_1,a_2,\ldots,a_n)$ such that $a_i\in A_i$ for all $i$, and no two of the $a_i$ are the same.
	Hall's Marriage Theorem plays a significant role in this paper, and it is as follows.
\begin{Thm}[\cite{brualdi1977introductory}]\label{thm:MC}
	The family $\mathcal{A}=(A_1,A_2,\ldots,A_n)$ of sets has a system of distinct representatives
	if and only if 
	\begin{center}
	\begin{minipage}{0.85\textwidth}
	(Hall's marriage condition)
for each $1\leq k\leq n$ and each choice of $k$ distinct indices $i_1,i_2,\ldots,i_k$ from $[n]$,
\[\left|A_{i_1} \cup A_{i_2} \cup \cdots \cup A_{i_k}\right| \geq k.\]
	\end{minipage}
	\end{center}
\end{Thm}

A digraph $D$ is {\it complete} if, for every pair $x$, $y$ of distinct vertices of $D$, both $(x,y) \in A(D)$ and $(y,x) \in A(D)$.
The complete digraph on $n$ vertices is denoted by $\overleftrightarrow{K}_n$.

The results from the previous studies to be used in this paper are as follows.

\begin{Prop}[\cite{plesnik1975graphs}]\label{prop:outdegree}
	Let $D$ be a $(t,\lambda)$-liking digraph.
	Then $|V(D)|\ge t+\lambda$ and $\delta^+(D) \ge t+\lambda-1$.
\end{Prop}

The following results are from \cite{choi2024generalizedlikingdigraph}.

\begin{Prop}\label{prop:eulerian}
	Let $D$ be a $(t,\lambda)$-liking digraph.
	If $t\ge \lambda+1$ or $d^+(v)=t+\lambda-1$ for each vertex $v$, then $d^+(v)=d^-(v)$ for every vertex $v$ in $D$.
\end{Prop}

\begin{Prop}\label{prop:out-degree}
Let $D$ be a $(t,\lambda)$-liking digraph.
For each $1\leq i \leq t-1$, every $t-i$ vertices have at least $\lambda+i$ common out-neighbors in $D$.
\end{Prop}

\begin{Thm}\label{thm:prev_main1}
If $t\ge \lambda+2$, then the complete digraph on $t+\lambda$ vertices is the only $(t,\lambda)$-liking digraph.
\end{Thm}

\begin{Thm}\label{thm:prev_main2}
Let $D$ be a $(t,\lambda)$-liking digraph for some positive integers $t, \lambda$ with $t\ge 2$.
Then the following are equivalent.
\begin{enumerate}[(a)]
	\item $D$ is complete on $t+\lambda$ vertices, that is, $D \cong \olr{K}_{t+\lambda}$.
	\item $D$ is a $(t-1,\lambda+1)$-liking digraph.
	\item $d^+(v)=t+\lambda-1$ for each vertex $v$.
\end{enumerate}
Furthermore, (a) is equivalent to the condition (d) if $(t, \lambda) \neq (2,1)$; (e) if $t\ge 3$, where
\begin{enumerate}[(a)]
	\item[(d)] there is a vertex $v$ satisfying $N^+(v)=V(D)\setminus \{v\}$;
	\item[(e)] $D$ is diregular.
\end{enumerate}

\end{Thm}

\section{Proofs of Theorems~\ref{thm:main1} and \ref{thm:main2}}

Given a digraph $D$, the {\it converse} of $D$, denoted by $D^{\leftarrow}$, is the digraph obtained from $D$ by reversing the direction of each arc in $D$.
Then, by the definition of a two-way $(t,\lambda)$-liking digraph, the following proposition is immediately true.

\begin{Prop}\label{prop:converse}
	The converse of a two-way $(t,\lambda)$-liking digraph is a two-way $(t,\lambda)$-liking digraph.
\end{Prop}

\begin{Lem}\label{lem:2-diregular}
		Let $D$ be a two-way $(t,\lambda)$-liking digraph of order $n$ with $\lambda \ge t \ge 2$.
		Then $D$ is $k$-diregular where  	\[{n-1 \choose t-1}{\lambda \choose t}={k \choose t}{\lambda-1 \choose t-1}. \qedhere \]
\end{Lem}
\begin{proof}
	Take a vertex $v$.
	We consider the set $S_v^{+}$ of ordered pairs $(X,Y)$ such that $X$ is a $t$-subset of $N^+(v)$ and $Y$ is a $(t-1)$-subset of $\left(\bigcap_{x\in X} N^-(x)\right)\setminus \{v\}$, that is,
	\[S_v^{+}:=\left\{(X, Y) \colon\, |X|=t, |Y|=t-1, X\subseteq N^+(v), Y\subseteq \left(\bigcap_{x\in X} N^-(x)\right) \setminus \{v\}\right\}. \]
	We may compute $|S_v^{+}|$ in two ways.
	First, fix $X \subseteq N^+(v)$ with $|X|=t$.
	There are ${d^{+}(v) \choose t}$ ways of doing so.
	Since $D$ is a two-way $(t,\lambda)$-liking digraph, the vertices in $X$ have exactly $\lambda$ common in-neighbors including $v$.
	Thus, given $X$, there are  ${\lambda-1 \choose t-1}$ ways to choose $Y$ so that $(X,Y) \in S_v^{+}$.
	Therefore
	\[|S^+_v|= {d^{+}(v) \choose t}{\lambda-1 \choose t-1}.\]
	To apply the second method, fix $Y \subseteq V(D)\setminus \{v\}$ with $|Y|=t-1$.
	There are ${n-1 \choose t-1}$ ways of choosing such a $Y$.
	Since any $t$ vertices in $D$ have exactly $\lambda$ common out-neighbors, 
	\[\left| N^+(v) \cap \bigcap_{y\in Y}N^+(y)\right|=\lambda. \]
	Thus there are ${\lambda \choose t}$ ways to choose $X$ so that $(X,Y) \in S_v^{+}$.
	Therefore \[|S_v^{+}|= {n-1 \choose t-1}{\lambda \choose t}.\]
	Hence, 
	\[{d^{+}(v) \choose t}{\lambda-1 \choose t-1}={n-1 \choose t-1}{\lambda \choose t}. \]
	By Proposition~\ref{prop:converse}, we may apply the above argument to the converse of $D$ to obtain the equation
	\[{d^{-}(v) \choose t}{\lambda-1 \choose t-1}={n-1 \choose t-1}{\lambda \choose t}. \]
	Thus	\[{d^{+}(v) \choose t}={d^{-}(v) \choose t}=\frac{{n-1 \choose t-1}{\lambda \choose t}}{{\lambda-1 \choose t-1}},  \]
	which implies that $d^+(v)$ and $d^-(v)$ are equal and expressed only in terms of $n$, $\lambda$, $t$.
	Therefore $d^+(v)=d^-(v)=k$ for some positive integer $k$ satisfying  	\[{n-1 \choose t-1}{\lambda \choose t}={k \choose t}{\lambda-1 \choose t-1}. \]
	Since $v$ was arbitrarily chosen, $D$ is $k$-diregular.\end{proof}
 	
 Lemma~\ref{lem:2-diregular} directly implies that for an integer $\lambda \ge 2$, a two-way $(2,\lambda)$-liking digraph is $k$-diregular for some positive integer $k$ satisfying $\lambda(n-1)=k(k-1)$. 
 Thus Theorem~\ref{thm:main1} is valid.
 We further study the existence of a $k$-diregular two-way $(2, \lambda)$-liking digraph by using an $(n,k,\lambda)$-SBIBD.

\begin{Prop}\label{prop:design implies}
	A $k$-diregular two-way $(2, \lambda)$-liking digraph of order $n$ exists if there is an $(n,k,\lambda)$-SBIBD with $n\ge 2\lambda$.
\end{Prop}
\begin{proof}
	Suppose that there is an $(n,k,\lambda)$-SBIBD with $n\ge 2\lambda$.
	By one of well-known facts for the parameters of a symmetric design,
	\[\lambda(n-1)=k(k-1). \]
	By the definition of a balanced incomplete block design,
	\begin{equation}\label{eq:n,k}
		k<n.
	\end{equation}
	Let $V=\{v_1, \ldots, v_{n}\}$ be the set of varieties of the $(n,k,\lambda)$-SBIBD and $\mathcal{B}=\{B_1, \ldots, B_{n}\}$ be the set of blocks of the $(n,k,\lambda)$-SBIBD.
	\begin{claim1} 
	The collection $\{V-B_1, \ldots, V-B_n\}$ has a system of distinct representatives.
	\end{claim1}
	\begin{proof}[Proof of Claim A]
	We will show that $\{V-B_1, \ldots, V-B_n\}$ satisfies Hall's marriage condition given in Theorem~\ref{thm:MC}.
	To this end, take an arbitrary nonempty subset $S \subseteq [n]$.
	Since 
	\[\left| \bigcup_{s \in S} (V-B_s) \right|=\left|V- \bigcap_{s \in S} B_s \right|, \]
	we only need to show 
	\[\left|\bigcap_{s \in S} B_s \right| \le n-|S|. \]
	If $|S|=1$, then, by \eqref{eq:n,k}, \[\left|\bigcap_{s \in S} B_s \right|=|B_i|=k\le n-1=n-|S|\]
	for some $i\in [n]$.
	If $1< |S| \le \lambda$, then, by the hypothesis that $n\ge 2\lambda$, \[\left|\bigcap_{s \in S} B_s \right|
	\le |B_i \cap B_j|=\lambda \le n-\lambda \le n-|S|\]
	for $\{i,j\} \subseteq S$ where the second inequality holds by the fact that any two blocks have exactly $\lambda$ varieties in common.
	If $\lambda<|S| \le k$,
	then, by \eqref{eq:n,k}, \[\left|\bigcap_{s \in S} B_s \right| \le  1 \le n-k \le n-|S|,\]
	where the first inequality holds by the fact that each pair of varieties simultaneously appears in exactly $\lambda$ blocks.
	If $k<|S| $,
	then \[\left|\bigcap_{s \in S} B_s \right|=0 \le n-|S|,\]
	where the first equality is true by the fact that each variety is contained in exactly $k$ blocks.
	Thus, no matter what the size of $S$ is, the following holds:
	\[\left|\bigcap_{s \in S} B_s \right| \le n-|S|. \]
	Therefore Hall's marriage condition is satisfied and so $\{V-B_1, \ldots, V-B_{n}\}$ has a system of distinct representatives by Theorem~\ref{thm:MC}.
	\end{proof}
By the above claim, $\{V-B_1, \ldots, V-B_{n}\}$ has a system of distinct representatives $\{v_{i_1}, \ldots, v_{i_n}\}$.
	Now, we consider the digraph $D$ with the vertex set $V$ and the arc set
	\[\bigcup_{t=1}^{n} \{(v,v_{i_t}) \colon\, v\in B_t\}.\]
	Since $v_{i_t}\not\in B_t$ for each $1\le t \le n$, $D$ is loopless.
	Further, since any two varieties in $V$ are contained in exactly $\lambda$ blocks, any two vertices in $D$ have precisely $\lambda$ common out-neighbors.
	Since any two blocks have exactly $\lambda$ varieties in common and $V(D)=\{v_{i_1}, \ldots, v_{i_n}\}$, any two vertices in $D$ have precisely $\lambda$ common in-neighbors.
	Thus $D$ is a two-way $(2,\lambda)$-liking digraph of order $n$. 
	Since each block has size $k$ and the design is symmetric, $D$ is $k$-diregular.
	\end{proof}

Now, we go further to derive more results for proving Theorem~\ref{thm:main2}.

\begin{Lem}\label{lem:lambda+1}
Let $D$ be a two-way $(\lambda+1,\lambda)$-liking digraph for some integer $\lambda \ge 2$.
	If $(u,v) \not \in A(D)$ or $(v,u) \not \in A(D)$ for some two vertices $u$ and $v$, then $d^+(u) = d^+(v)$.
\end{Lem}
\begin{proof}
	Without loss of generality, we may assume that there are two vertices $u$ and $v$ with $(u,v)\not \in A(D)$.
	That is, \[u \not\in N^-(v)\quad \text{and} \quad v \not\in N^+(u).\]
	Take a $\lambda$-subset $S$ in $N^+(u)$.
	Since $v \not\in N^+(u)$, $|S\cup \{v\}|=\lambda+1$.
	Thus there is a unique $\lambda$-subset $T$ in $N^-(v)$ in which each vertex is a common in-neighbor of the vertices in $S\cup \{v\}$.
	Since $u \not\in N^-(v)$, $u \not\in T$.
	We let $f(S)=T$.
	Then $f$ is a function from the family $\mathcal{F}^+$ of $\lambda$-subsets in $N^+(u)$ to the family $\mathcal{F}^-$ of $\lambda$-subsets in $N^-(v)$.
	Suppose that $f(S')=T$ for a $\lambda$-subset $S'$ in $N^+(u)$ distinct from $S$.
	Then the vertices in $T\cup \{u\}$ are common in-neighbors of the vertices in $S\cup S'$.
	However, $|S\cup S'| \ge \lambda+1$ and $|T\cup \{u\}|=\lambda+1$, so we reach a contradiction to the fact that $D$ is a two-way $(\lambda+1,\lambda)$-liking digraph.
	Hence $f$ is one-to-one.
	Thus we have 
	\[{d^+(u) \choose \lambda}=|\mathcal{F}^+| \le |\mathcal{F}^-|={d^-(v) \choose \lambda} \]
	and so $d^+(u) \le d^-(v)$.
	Since a two-way $(\lambda+1,\lambda)$-liking digraph is a $(\lambda+1,\lambda)$-liking digraph, $d^-(v)=d^+(v)$ by Proposition~\ref{prop:eulerian}. 
	Therefore \[d_D^+(u) \le d_D^+(v).\] 
	
	Now, we consider the converse $D^{\leftarrow}$ of $D$.
	Then $(v,u) \not \in A(D^{\leftarrow})$.
	By Proposition~\ref{prop:converse}, $D^{\leftarrow}$ is a two-way $(\lambda+1, \lambda)$-liking digraph.
	Thus we may apply the same argument as above to obtain 
	 \[d_{D^{\leftarrow}}^+(v) \le d_{D^{\leftarrow}}^+(u),\]
	i.e.
	 \[d_{D}^-(v) \le d_{D}^-(u).\]
	Therefore $d_{D}^+(v) \le d_{D}^+(u)$ by Proposition~\ref{prop:eulerian}.
	Hence $d_D^+(u)=d_D^+(v)$. 
\end{proof}

Given a digraph $D$ of order $n$, the {\it complement} of $D$, denoted by $\overline{D}$, is the digraph with $V(\overline{D})=V(D)$ and $A(\overline{D})=A(\overleftrightarrow{K}_{n})-A(D)$.

\begin{Prop}\label{prop:lambda+1_complete}
The complete digraph on $2\lambda+1$ vertices is the only two-way $(\lambda+1,\lambda)$-liking digraph for an integer $\lambda \ge 2$.
\end{Prop}
\begin{proof}
	Let $D$ be a two-way $(\lambda+1,\lambda)$-liking digraph for some integer $\lambda \ge 2$.
	We first consider the case in which the complement $\overline{D}$ of $D$ is weakly connected.
	Take two vertices $u$ and $v$.
	Since $\overline{D}$ is weakly connected, there is a $(u,v)$-path $P$ in $U(\overline{D})$ where $U(\overline{D})$ denotes the underlying graph of $\overline{D}$.
	This implies that for any two consecutive vertices $x$ and $y$ on $P$, $(x,y) \not\in A(D)$ or $(y,x)\not\in A(D)$.
	Thus we may apply Lemma~\ref{lem:lambda+1} to conclude that $d^+(u)=d^+(v)$.
	By the way, by Proposition~\ref{prop:eulerian}, $d^-(u)=d^+(u)=d^+(v)=d^-(v)$.
	Since $u$ and $v$ were arbitrarily chosen, we may conclude that $D$ is diregular.
	Therefore $D$ is the complete digraph on $2\lambda+1$ vertices by Theorem~\ref{thm:prev_main2}$(e)\Rightarrow (a)$.
	
	Now we consider the case in which the complement $\overline{D}$ of $D$ is not weakly connected.
	Then the vertex set $V(D)$ can be partitioned into two subsets $X$ and $Y$ such that there is no arc in $\overline{D}$ between a vertex in $X$ and a vertex in $Y$.
	Thus for any $x\in X$ and $y \in Y$,
	\begin{equation}\label{eq:2}
		(x,y)\in A(D) \quad \text{and} \quad (y,x) \in A(D). 
	\end{equation}
	Furthermore, since $|V(D)|\ge 2\lambda+1$ by Proposition~\ref{prop:outdegree} and $V(D)=X \cup Y$, one of $X$ and $Y$ has at least $\lambda+1$ vertices.
	Without loss of generality, we may assume that $|X| \ge \lambda+1$.
	Then any $\lambda+1$ vertices in $X$ have the vertices in $Y$ as common out-neighbors by \eqref{eq:2}.
	Thus, since $D$ is a two-way $(\lambda+1,\lambda)$-liking digraph, \begin{equation}\label{eq:Ysize}
		|Y| \le \lambda.
	\end{equation}
	Moreover, any $\lambda+1$ vertices in $X$ have exactly $\lambda-|Y|$ common out-neighbors in $X$.
	Then $D[X]$ is a $(\lambda+1, \lambda-|Y|)$-liking digraph.

	To the contrary, suppose that neither $D[X]$ nor $D[Y]$ are complete.
		If $1\le |Y| <\lambda$, then \[\lambda+1 \ge (\lambda-|Y|)+2\] 
		and so the $(\lambda+1, \lambda-|Y|)$-liking digraph $D[X]$ is complete by Theorem~\ref{thm:prev_main1}, which is a contradiction.
	Thus $|Y| \ge \lambda$.
	Then $|Y|=\lambda$ by \eqref{eq:Ysize}.
	Take a vertex $x$ in $X$.
	Then, since \[N^+(x) \cap \bigcap_{y\in Y} N^+(y)=N^+(x)\cap X\] by \eqref{eq:2} and $|\{x\} \cup Y|=\lambda+1$, 
	\[|N^+(x) \cap X|=\lambda. \] 
	Since $D[Y]$ is not complete, there are vertices $y$ and $y'$ in $Y$ such that \[(y',y) \not \in A(D).\]	
	Then 
	\begin{align*}\label{eq:4}
		N^+(x) \cap N^+(y') 
		&= (N^+(x) \cap N^+(y') \cap X) \cup  (N^+(x) \cap N^+(y') \cap Y) \notag \\
		&= (N^+(x) \cap X) \cup  (N^+(y') \cap Y) \notag \\ 
		&\subseteq (N^+(x) \cap X) \cup  (Y\setminus\{y,y'\}). 	
	\end{align*}
	where the second equality holds by \eqref{eq:2}, which implies $X \subseteq N^+(y')$ and $Y\subseteq N^+(x)$.
	Thus \[|N^+(x) \cap N^+(y')| \le |N^+(x) \cap X|+|Y\setminus\{y,y'\}|=\lambda+(\lambda-2)=2\lambda-2. \]
	However, by Proposition~\ref{prop:out-degree}, $x$ and $y'$ have at least $2\lambda-1$ common out-neighbors.
	Therefore we have reached a contradiction.
	Thus $D[X]$ or $D[Y]$ is complete.
	Then there is a vertex in $D$ having outdegree $|V(D)|-1$ by \eqref{eq:2} and so $D$ is the complete digraph on $2\lambda+1$ vertices by Theorem~\ref{thm:prev_main2}$(d)\Rightarrow(a)$.
	\end{proof}

\begin{proof}[Proof of Theorem~\ref{thm:main2}]
Let $D$ be a two-way $(t,\lambda)$-liking digraph with $t\ge 3$.
If $\lambda \ge t$, then $D$ is diregular by Lemma~\ref{lem:2-diregular} and so $D$ is the complete digraph on $t+\lambda$ vertices by Theorem~\ref{thm:prev_main2}$(e)\Rightarrow(a)$.
If $t\ge \lambda+2$, then $D$ is the complete digraph on $t+\lambda$ vertices by Theorem~\ref{thm:prev_main1}.
If $t=\lambda+1$, then $D$ is the complete digraph on $t+\lambda$ vertices by
Proposition~\ref{prop:lambda+1_complete}.
Therefore we may conclude that $D$ is the complete digraph on $t+\lambda$ vertices.
\end{proof}

\section{An extension of results on $(t,\lambda)$-liking digraphs}

\begin{Prop}\label{prop:lambda+1_converse}
	A $(\lambda+1,\lambda)$-liking digraph is a two-way $(\lambda+1,\lambda)$-liking digraph for a positive integer $\lambda$.
\end{Prop}
\begin{proof}
	Let $D$ be a $(\lambda+1,\lambda)$-liking digraph for a positive integer $\lambda$.
	To prove that $D$ is a two-way $(\lambda+1, \lambda)$-liking digraph, we first show that
	\[ \sum_{\substack{S \subseteq V(D) \\ |S|=\lambda+1}} \left|\bigcap_{v\in S}N^-(v)\right|=\sum_{\substack{S \subseteq V(D) \\ |S|=\lambda+1}} \lambda.  \]
	Noting that 
	\[\left|\bigcap_{v\in S}N^-(v)\right|=\sum_{w\in \bigcap_{v\in S}N^-(v)} 1, \]
	for a vertex set $S$ of $D$,
	we may change the order of the double summation as follows:
	\[\sum_{\substack{S \subseteq V(D) \\ |S|=\lambda+1}} \left|\bigcap_{v\in S}N^-(v)\right| =\sum_{w\in V(D)}\sum_{\substack{S \subseteq N^+(w) \\ |S|=\lambda+1}} 1.\]
	By Proposition~\ref{prop:eulerian}, $d^+(w)=d^-(w)$ for each vertex $w$ in $D$ and so
	\[\sum_{w\in V(D)}\sum_{\substack{S \subseteq N^+(w) \\ |S|=\lambda+1}} 1=\sum_{w\in V(D)}\sum_{\substack{S \subseteq N^-(w) \\ |S|=\lambda+1}} 1 .\]
	By changing the order of the double summation, we obtain
	\[\sum_{w\in V(D)}\sum_{\substack{S \subseteq N^-(w) \\ |S|=\lambda+1}} 1=\sum_{\substack{S \subseteq V(D) \\ |S|=\lambda+1}} \left|\bigcap_{v\in S}N^+(v)\right|. \]
	Since $D$ is a $(\lambda+1,\lambda)$-liking digraph, 
	\[\sum_{\substack{S \subseteq V(D) \\ |S|=\lambda+1}} \left|\bigcap_{v\in S}N^+(v)\right|=\sum_{\substack{S \subseteq V(D) \\ |S|=\lambda+1}} \lambda. \]
	Accordingly, we have shown that 
	\[ \sum_{\substack{S \subseteq V(D) \\ |S|=\lambda+1}} \left|\bigcap_{v\in S}N^-(v)\right|=\sum_{\substack{S \subseteq V(D) \\ |S|=\lambda+1}} \lambda .  \]	
	By the way, since $D$ is a $(\lambda+1,\lambda)$-liking digraph,
	 for any ($\lambda+1$)-subset $S$ of $V(D)$, 
	\[\left|\bigcap_{v\in S}N^-(v)\right| \le \lambda. \]
	Therefore, for any ($\lambda+1$)-subset $S$ of $V(D)$, 	\[\left|\bigcap_{v\in S}N^-(v)\right| = \lambda. \]
	That is, any $\lambda+1$ vertices in $D$ have exactly $\lambda$ common in-neighbors.
	Hence, $D$ is a two-way $(\lambda+1,\lambda)$-liking digraph.
\end{proof}

In \cite{choi2023digraph} and \cite{choi2024generalizedlikingdigraph}, it was shown that when $t$ and $\lambda$ satisfy $(t,\lambda)=(2,1)$ or $t\ge \lambda+2$, a digraph is a $(t,\lambda)$-liking digraph if and only if it is a two-way $(t,\lambda)$-liking digraph.
Thus Proposition~\ref{prop:lambda+1_converse} extends the results on $(t,\lambda)$-liking digraphs as follows.

\begin{Cor}\label{cor:lambda+1_twoway}
	Let $t$ and $\lambda$ be positive integers satisfying $t\ge \lambda+1$. 
	Then a digraph is a $(t,\lambda)$-liking digraph if and only if it is a two-way $(t,\lambda)$-liking digraph.
\end{Cor}

The lower bound of the above corollary is tight.
In \cite{choi2024generalizedlikingdigraph}, there is an example that is a $(2,2)$-liking digraph of order $7$ which is not diregular.
It is easy to check that this digraph is not a two-way $(2,2)$-liking digraph.

Now, Theorem~\ref{thm:main2} and Corollary~\ref{cor:lambda+1_twoway} extend Theorem~\ref{thm:prev_main1}.

\begin{Thm}\label{thm:lambda+1_complete}
		If $t\ge \lambda+1$ and $t\ge 3$, then the complete digraph on $t+\lambda$ vertices is the only $(t,\lambda)$-liking digraph.
\end{Thm}

\section{Acknowledgement}
This work was supported by Science Research Center Program through the National Research Foundation of Korea(NRF) grant funded by the Korean Government (MSIT)(NRF-2022R1A2C1009648).
%

\end{document}